\documentclass[11pt]{amsart}
\usepackage{amsmath,amssymb}
\usepackage[curve,matrix,arrow]{xy}

\newtheorem{thm}{Theorem}[section]

\newtheorem{lemma}[thm]{Lemma}
\newtheorem{prop}[thm]{Proposition}

\theoremstyle{definition}
\newtheorem{defn}[thm]{Definition}
\newtheorem{remark}[thm]{Remark}
\newtheorem{example}[thm]{Example}

\newcommand\wh[1]{\widehat{#1}}
\newcommand{\ovl}[1]{\overline{#1}}
\newcommand{\ls}[2]{{^{#1}\!{#2}}}

\newcommand{\qbox}[1]{\quad\hbox{#1}\quad}
\def\dst{\displaystyle}
\def\noo{\trianglelefteq_o}

\def\onto{\twoheadrightarrow}

\def\Inf{\operatorname{Inf}\nolimits}
\def\Hom{\operatorname{Hom}\nolimits}

\def\spec{\operatorname{Spec}\nolimits}

\newcommand\prlim[1]{{\underset{{#1}}{\varprojlim}}}

\def\Ob{\operatorname{\bf Ob}}
\def\set{\operatorname{\bf Set}}
\def\xset{\ls{\times}{\operatorname{\bf Set}}}
\def\Mack{\operatorname{\bf Mack}}

\def\AF{\mathcal AF}
\def\xAF{\ls{\times}{\!\mathcal A}F}

\def\xB{\ls{\times}B}

\def\Ab{\mathcal Ab}
\def\fix{\operatorname{Fix}}
\def\xfix{\ls{\times}{\operatorname{Fix}}}

\def\FP{\operatorname{FP}}
\def\FQ{\operatorname{FQ}}

\def\Z{\mathbb Z}
\def\N{\mathbb N}
\def\Q{\mathbb Q}

\def\CO{\mathcal O}

\def\G{\mathbf G}
\def\cl{\operatorname{Cl}}

\begin{document}

\title{On the (crossed) Burnside ring of profinite groups}
\author{Nadia Mazza}
\address
{Department of Mathematics and Statistics\\ Lancaster University\\ Lancaster \\LA1 4YF, UK} 
\email{n.mazza@lancaster.ac.uk}
\keywords{(crossed) Burnside ring, profinite groups, FC-groups, Mackey functors}

\subjclass[2020]{Primary: 19A22, 20E18, Secondary: 20F24}

\begin{abstract} In this paper we investigate some properties of the Burnside ring of a profinite group as defined in \cite{ds}.  
We introduce the notion of the crossed Burnside ring of a profinite FC-group, and generalise some results from finite to profinite (FC-)groups.
In our investigations, we also obtain results on profinite FC-groups which may be of independent interest.
\end{abstract}

\maketitle

\section{Introduction}\label{sec:intro}

The Burnside ring $B(G)$ of a group $G$ is a commutative ring, which encodes in some way useful information about the abstract group $G$ (e.g. \cite{dress-sol}). 
In particular, if $G$ is a finite group, the mapping of a subgroup $H$ of $G$ to the underlying abelian group of the Burnside ring $B(H)$ defines a projective Mackey functor for $G$, and this fact is key in the study of Mackey functors for finite groups \cite{tw}. 
The crossed $G$-sets of a finite group $G$ act on the category of Mackey functors for $G$, leading to a decomposition of the Mackey algebra of $G$ into $p$-blocks, after extending the scalars to some suitable $p$-local ring for a given prime $p$ dividing the order of $G$ \cite{b-mackey}.

In \cite{dress-sol}, Dress proves that, for $G$ finite, there exists a 1-1 correspondence between the connected components of the prime ideal spectrum of $B(G)$ and the conjugacy classes of perfect subgroups of $G$. 
In \cite{gluck}, Gluck gives a formula to calculate the primitive idempotents of $B(G)$, and he uses his result to provide an algebraic proof of Brown's result on the Euler characteristic of the simplicial complex whose vertices are the nontrivial $p$-subgroups of $G$.

In \cite{ds}, the authors introduce the Burnside ring $\wh B(G)$ of a profinite group $G$ as a generalisation of the Burnside ring of a finite group. 
In \cite[Section 5]{ds}, the authors hint at certain properties of $\wh B(G)$, which are similar to those of Mackey functors according to Dress \cite[Section 2]{tw}. 
Their results have been used in \cite{bb} to study Mackey functors arising in number theory.

The main objective of the present paper is to investigate a generalisation to profinite groups of the above results on the (crossed) Burnside ring of a finite groups and its applications.
Hence, in Sections~\ref{sec:recap} and~\ref{sec:profinite}, we review the known background on crossed Burnside rings for finite groups and on the Burnside ring of a profinite group. 
These lead us to take a closer look at profinite FC-groups in Section~\ref{sec:fc-gp}. 
An {\em FC-group} is a group in which every element has a finite conjugacy class. That is, a group $G$ is FC if and only if $C_G(g)$ has finite index in $G$, for all $g\in G$. 
Properties of FC-groups are described in \cite{tomk}, and also in \cite[Section 14.5]{Rob}. In particular, finite groups, and profinite abelian groups are profinite FC. 
We make a few observations about the structure of profinite FC-groups, which we have not found in the literature, and which may be of independent interest.
In Section~\ref{sec:xb}, we use Dress and Siebeneicher's construction of the Burnside ring of a profinite group and define the crossed Burnside ring \cite{OY1} of profinite FC-groups.
In Section~\ref{sec:spec}, we generalise in some way Dress and Gluck's results on the idempotents of the Burnside $\Q$-algebra and Burnside ring of a finite group to the class of profinite groups.
Finally, in Section~\ref{sec:xb-mackey}, we turn to Mackey functors, and review the approaches in \cite{bb,ds}, before generalising Oda and Yoshida's results \cite{b-mackey, OY1} to obtain an action of almost finite crossed $G$-spaces on the category of Mackey functors.


\section{Background on the crossed Burnside ring of finite groups}\label{sec:recap}

We recall the needed background on crossed Burnside rings of finite groups from \cite[Section 2]{OY1}.
Let $G$ be a finite group and let $\set_G$ denote the category whose objects are the finite $G$-sets and the morphisms are $G$-equivariant maps. 
Let $S$ be a normal subgroup of $G$, which we regard as a $G$-set for the conjugation action: $(g,s)\mapsto\ls gs=gsg^{-1}$ for all $g\in G$ and all $S\in S$. 
Then $S$ is a $G$-monoid. That is, $S$ is a $G$-set equipped with a multiplication $S\times S\to S$ for which there is a multiplicative identity $1_S$. 

\begin{defn}\cite[(2.6,2.7)]{OY1}
Let $S$ be a normal subgroup of $G$. A {\em crossed $G$-set over $S$} is a morphism $f:X\to S$ in $\set_G$.

Given two crossed $G$-sets over $S$, say $f_i:X_i\to S$, for $i=1,2$, their sum and product are the crossed $G$-sets:
$$f_1+f_2~:~X_1\sqcup X_2~\longrightarrow S\qbox{and}f_1\times f_2~:~X_1\times X_2~\longrightarrow S,$$
where
$$(f_1+f_2)(x)=f_i(x)\qbox{for $x\in X_i$, for $i=1,2$, and}(f_1\times f_2)(x_1,x_2)=f_1(x_1)f_2(x_2).$$
The additive identity element is the unique morphism $\emptyset\to S$, where $\emptyset$ is the empty set (i.e. the initial object in $\set_G$), and the multiplicative identity is $u:G/G\to S$, where $u(G)=1_S$.
Addition and multiplication are commutative up to isomorphism. In particular,
$$\begin{array}{rcl}
\big(f_1\times f_2~:~X_1\times X_2~\longrightarrow S\big)&\longrightarrow&\big(f_2\times f_1~:~X_2\times X_1~\longrightarrow S\big)\\
(f_1\times f_2)(x_1,x_2)&\longmapsto&(f_2\times f_1)(f_1(x_1)x_2,x_1),\end{array}$$ 
is an isomorphism of crossed $G$-sets since 
$$f_2\big(f_1(x_1)x_2\big)f_1(x_1)=\ls{f_1(x_1)}f_2(x_2)f_1(x_1)=f_1(x_1)f_2(x_2)\qbox{in $S$.}$$

Define the category $\xset_{(G,S)}$ to be the category whose objects are the crossed $G$-sets over $S$, and the morphisms 
$$\phi:(f_1:X_1\to S)\longrightarrow(f_2:X_2\to S)\qbox{in $\xset_{(G,S)}$}$$
are the $G$-equivariant maps $\phi:X_1\to X_2$ such that $f_2\phi=f_1:X_1\to S$.
\end{defn}

The category of crossed $G$-sets over $S$ is a commutative monoid. 
The Grothendieck construction \cite[Section 24.1]{may} turns the abelian monoid of isomorphism classes of crossed $G$-sets over $S$ into a commutative ring $\xB(G,S)$, called the {\em crossed Burnside ring} of $G$ over $S$ .
That is, the elements of $\xB(G,S)$ are the isomorphism classes of virtual crossed $G$-sets over $S$, which can be written as differences 
$$[f_1:X_1\to S]-[f_2:X_2\to S],$$
where $[f_i:X_i\to S]$ are isomorphism classes of crossed $G$-sets over $S$.

Unless otherwise stated, we will henceforth take $S=G$ as $G$-monoid with conjugation action of $G$, and we denote it $G^c$ to avoid any confusion. 
Thus, $G^c=\dst\bigsqcup_{x\in\cl(G)}G/C_G(x)$, where $\cl(G)$ is a set of representatives of the conjugacy classes of the elements of $G$. 
Then, we let $B^c(G)$ denote the crossed Burnside ring of $G$ over $G^c$ and simply call it the {\em crossed Burnside ring of $G$}.
 
As a group, $B^c(G)$ is free abelian with basis the isomorphism classes of transitive crossed $G$-sets. These have the form $[w_a:G/H\to G^c]$, where $w_a(gH)=\ls ga$ for some $a\in C_G(H)$.
We have $(w_a:G/H\to G^c)\cong(w_b:G/K\to G^c)$ as crossed $G$-sets if and only if $K=\ls gH$ and $b=\ls ga$ for some $g\in G$.
The Burnside ring $B(G)$ of $G$ embeds into $B^c(G)$ via the injective ring homomorphism: 
$[G/H]\mapsto[w_1:G/H\to G^c]$. We refer to \cite{b-mackey,OY1} for further properties of $B^c(G)$ for a finite group.


\section{From finite to profinite}\label{sec:profinite}

Let $G$ be a profinite group.
By a {\em subgroup} of $G$, we mean a {\em closed subgroup} of $G$.
If $U$ is an open (normal) subgroup of $G$, we write $U\leq_oG$ ($U\trianglelefteq_oG$).
We refer the reader to \cite{RZ,wilson} for the background on profinite groups. 

\vspace{.3cm}
We recall the definition of the Burnside ring of a profinite group and the basic concepts introduced in \cite[Section 2]{ds}, referring the interested reader to that article for the details. 

A {\em $G$-space} is a Hausdorff topological space $X$ equipped with a continuous $G$-equivariant action $\rho:G\times X\to X$. 
For $x\in X$, its stabiliser is the closed subgroup $G_x=\{g\in G\mid gx=x\}$ of $G$ and its orbit is the closed compact subset $Gx=\{gx\mid g\in G\}$ of $X$. 
Throughout, we denote $G\backslash X$ the set of $G$-orbits of $X$, and $[G\backslash X]$ a set of representatives.

We call $X$ {\em essentially finite} if the fixed point sets $|X^U|$ are finite for all the open subgroups $U$ of $G$.
A $G$-space $X$ is {\em almost finite} if $X$ is an essentially finite discrete topological space.

Given two essentially finite $G$-spaces $X,Y$, we define an equivalence relation
$$X\sim Y\qbox{if and only if}|X^U|=|Y^U|,\;\forall\;U\leq_oG,$$
where $X^U=\{x\in X\mid ux=x,\;\forall\; u\in U\}$. It follows that two almost finite $G$-spaces are equivalent if and only if they are isomorphic.
If $X$ is an essentially finite $G$-space, then the equivalence class $[X]$ of $X$ contains an almost finite $G$-space which is unique up to isomorphism.
In other words, considering equivalence classes of essentially finite $G$-spaces is the same as considering isomorphism classes of almost finite $G$-spaces.
Observe that if $X$ is an essentially finite $G$-space, then for all $U\leq_oG$, the set of $U$-fixed points $X^U$ is a finite $N_G(U)/U$-set.

Suppose that $X$ is a discrete $G$-space and write $X=\sqcup_{x\in[G\backslash X]}Gx$ as the disjoint union of its $G$-orbits. 
Since $G$ is compact, every orbit is a compact discrete $G$-space, i.e. finite. 
It follows that the bijection from the coset space $G/G_x$ to $Gx$, defined by $gG_x\mapsto gx$, is a homeomorphism.
Hence, a discrete $G$-space $X$ is almost finite if 
\begin{equation}\label{eq:basic}
X\cong\bigsqcup_{x\in[G\backslash X]}G/G_x,\qbox{where}G_x\leq_oG,\;\hbox{and}
\end{equation}
for all $U\leq_oG$, there exist finitely many orbits $Gx$ with $U$ contained in a $G$-conjugate of $G_x$.

\begin{defn}\label{def:afg}
Let $\AF_G$ be the category of {\em almost} {\em finite $G$-spaces}. 
The objects are the almost finite $G$-spaces, and the morphisms $f:X\to Y$ between two almost finite $G$-spaces $X$ and $Y$ are the $G$-equivariant maps (necessarily continuous). 
We write $\Hom_{\AF_G}(X,Y)$ for the set of morphisms $X\to Y$.
\end{defn}

Similarly to the case of finite groups, if $X$ and $Y$ are almost finite $G$-spaces, then $f:X\to Y$ can be expressed as 
$$\big(f_{x,y}\big)_{\!x,y}\qbox{where $(x,y)$ runs through $[G\backslash X]\times[G\backslash Y]$}$$
and $f_{x,y}:G/G_x\to G/G_y$ is of the form $f_{x,y}(G_x)=gG_y$ for some $g\in G$ such that $G_x\leq\ls gG_y$.
In particular, $G/U\cong G/V$ as almost finite $G$-spaces if and only if $U$ and $V$ are $G$-conjugate.

The isomorphism classes of almost finite $G$-spaces form an abelian monoid, with addition given by disjoint unions, and multiplication given by the cartesian product. Recall that
$$(X\times Y)^U=X^U\times Y^U,\quad\forall\;U\leq_oG,\;\forall\;X,Y\in\Ob(\AF_G).$$

\begin{defn}\label{def:cbr}
The {\em Burnside ring $\wh B(G)$} of a profinite group $G$ is the Grothendieck ring of the category $\AF_G$. The elements are the isomorphism classes of virtual almost finite $G$-spaces.
In $\wh B(G)$, we have $1=[G/G]$ and $0=[\emptyset]$, where $[X]$ denotes the isomorphism class of an almost finite $G$-space $X$. 
\end{defn}

Every element of $\wh B(G)$ can be written as a difference $[X]-[Y]$ of the isomorphism class of two almost finite $G$-spaces.
For convenience, we will often make the abuse of notation and omit the brackets to indicate elements of $\wh B(G)$.

In \cite{ds}, the authors show that $\wh B(G)\cong\prlim{N\noo G}B(G/N)$ is a complete topological commutative ring, generated by the isomorphism classes of transitive almost finite $G$-spaces.
We now want to generalise their results to introduce a crossed Burnside ring \cite{OY1} for profinite groups. 
In order to do so, we first want to find a suitable class of profinite groups where a similar construction works. 
Following the same approach as for finite groups, given a profinite group $G$, let $G^c$ denote the $G$-space on which $G$ acts by conjugation. 
We have a decomposition
$$G^c\cong\bigsqcup_{g\in\cl(G)}\ls Gg,$$
where $\cl(G)$ denotes a set of representatives of the conjugacy classes $\ls Gg=\{\ls ug\mid u\in G\}$ of $G$.
The topology on $G^c$ is induced by the subspace topology on each $\ls Gg$. In particular,
\begin{itemize}
\item
$G^c$ is discrete if and only if $|\ls Gg|<\infty$, i.e. if and only if $C_G(g)\leq_oG$ for all $g\in G$.
\item
$G^c$ is essentially finite if and only if $|(G^c)^U|=|C_G(U)|<\infty$ for all
$U\leq_o G$. 
\end{itemize}
These observations lead us to focus on the class of profinite FC-groups.


\section{FC-groups}\label{sec:fc-gp}

\begin{defn}\label{def:fc}
An {\em FC-group} is a group $G$ whose elements have finitely many conjugates. Equivalently, $|G:C_G(g)|<\infty$ for every $g\in G$.
\end{defn}
The term FC means {\em finite conjugacy (classes)}. The class of FC-groups is closed under taking subgroups, finite products and intersections, and quotients. It obviously contains all the abelian groups and all the finite groups.
FC-groups are a subclass of the class of groups with {\em restricted centralisers}, that is, groups in which the centralisers of elements are either finite or of finite index (cf. \cite{shalev}).

If $G$ is FC, the centraliser $C_G(H)=\displaystyle\bigcap_{1\leq i\leq n}C_G(h_i)$ of a finitely generated subgroup $H=\langle h_1,\dots,h_n\rangle$ of $G$ is the intersection of finitely many subgroups of finite index in $G$, and therefore $C_G(H)$ has finite index in $G$ too. 

From \cite[Section 1]{hallp}, we know that if $G$ is a torsion FC-group, then $G$ is locally finite (i.e. every finitely generated subgroup is finite). It follows that $G/Z(G)$ and $G'$ are locally finite for any FC-group $G$. In particular, if $G$ is finitely generated then $|G/Z(G)|$ and $|G'|$ are finite.
In \cite[Lemma 2.6]{shalev}, the author proves that if $G$ is a profinite FC-group, then $G'$ is finite, improving on the previous result, stating that $G'$ is a torsion group.
Therefore, a profinite FC-group is finite-by-abelian.
Recall that in a profinite group, $G'=\ovl{[G,G]}=\bigcap_{N\noo G}[G,G]N$ is the closure of the derived subgroup of $G$.

\vspace{.3cm}
As observed above, the centralisers of finitely generated subgroups of FC-groups have finite index. 
What can we say about the centraliser of a closed subgroup of an FC-group in general?

\begin{prop}\label{prop:cgu}
Let $G$ be an FC-group. TFAE
\begin{itemize}
\item[(i)] $Z(G)$ has finite index in $G$.
\item[(ii)] $\forall\;U\leq G$ of finite index, $C_G(U)$ has finite index in $G$. 
\item[(iii)] $\exists\;U\leq G$ of finite index such that $C_G(U)$ has finite index in $G$. 
\end{itemize}
By contrapositive, $Z(G)$ is a subgroup of infinite index in $G$, if and only if the centraliser of each subgroup of $G$ of finite index is itself a subgroup of infinite index in $G$.
\end{prop}
Note that in (iii), it is equivalent to assume that such $U$ is a normal subgroup of finite index in $G$ (up to replacing $U$ with its core in $G$).

\begin{proof}
(i)$\Rightarrow$(ii)$\Rightarrow$(iii) are obvious. To show (iii) implies (i), we pick a transversal $\{t_1,\dots,t_n\}$ of $U$ in $G$. Then,
$$Z(G)=C_G(U)\cap\bigcap_{1\leq i\leq n}C_G(t_i)\qbox{has finite index in $G$,}$$
since it is a finite intersection of subgroups of finite index in $G$. The proposition follows.
\end{proof}

Now, if $G$ is profinite FC, Shalev's result leads to the following.

\begin{prop}\label{prop:virtab}
Let $G$ be a profinite FC-group. Then $Z(G)$ is an open subgroup of $G$. In particular, $G$ is virtually abelian.
More generally, if $G$ is a residually finite FC-group whose derived subgroup is finite, then $Z(G)$ is a subgroup of finite index in $G$.
\end{prop}

As a consequence of Proposition~\ref{prop:virtab}, if $G$ is profinite FC, then the centraliser of any subgroup of $G$ is open in $G$. 

\begin{proof}
Since $G$ is residually finite, for each $x\in G'$, there exists a normal subgroup $U_x\triangleleft G$ of finite index in $G$ such that $x\notin U_x$. 
Set $U=\dst\bigcap_{x\in G'}U_x$. Then, $U\triangleleft G$ has finite index in $G$, and $U\cap G'=1$. Moreover, $[G,U]\leq G'\cap U=1$ shows that $U$ is a central subgroup. 
The result follows.
\end{proof}

\smallskip
Note that the profinite completion of an FC-group need not be an FC-group, as seen on a variant of P. Hall's example \cite[Example 2.1]{tomk}.

\begin{example}\label{ex:hall}
Let $p$ be a prime. For each $n\in\Z$, let:
$$X_n=\langle x_n,y_n,z_n~|~x_n^p=y_n^p=z_n^p=1,~[x_n,y_n]=z_n\rangle\cong p^{1+2}_+$$
be an extraspecial $p$-group of order $p^3$ and exponent $p$. For every $n\in\Z$, set
$$g_{2n-1}=x_{2n-1}x_{2n}\qbox{and}g_{2n}=y_{2n}y_{2n+1},\qbox{and define}G=\langle g_n~|~n\in\Z\rangle.$$
By definition, $[g_{2n-1},g_{2n}]=z_{2n}$ and $[g_{2n},g_{2n+1}]=z_{2n+1}^{-1}$, with $[g_i,g_j]=1$ whenever
$|j-i|\geq2$.
We have $G'=Z(G)=\langle z_n~|~n\in\Z\rangle$ is an infinite elementary abelian $p$-group, and $G/G'$ too. Moreover, $G$ has exponent $p$ and is nilpotent of class $2$.

For $g\in G$, let $\{g^G\}$ denote its conjugacy class, and for a subgroup $H$ of $G$, let $\langle H^G\rangle$ denote its normal closure. We have
$$\{g_n^G\}=\{g_nz_n^iz_{n+1}^j~|~0\leq i,j<p\}\qbox{and}\langle H^G\rangle\leq HZ(G).$$

Every element of the abstract group $G$ can be written in a unique way as a finite product $w=g_{i_1}^{a_1}\cdots g_{i_n}^{a_n}z_{j_1}^{b_1}\cdots z_{j_m}^{b_m}$ for some integers $i_1<\dots<i_n$ and $j_1<\dots<j_m$, and for integers
$0<a_1,\dots,a_n,b_1,\dots,b_m<p$.
We calculate $C_G(w)=\langle g_l\mid |l-i_s|>1\;,\;\forall\;1\leq s\leq n\rangle Z(G)$, and note that $C_G(w)$ is a normal subgroup of $G$ of finite index. 
Therefore, as an abstract group, $G$ is FC since the conjugacy class of any word $g_{i_1}^{e_1}\cdots g_{i_k}^{e_k}$ is finite.
However, in the profinite completion $\hat G$ of $G$, the elements which cannot be expressed as words of finite length in the $g_i$'s have conjugacy classes of infinite size. 
(For instance the conjugacy class of the element whose image in every finite quotient is the image of $g_1g_2g_3\cdots$ is infinite.)

\end{example}


\section{The crossed Burnside ring for profinite FC-groups}\label{sec:xb}

Let $G$ be a profinite group, and let $G^c$ denote the $G$-space on which $G$ acts by conjugation. 
Since
$$
(G^c)^U=C_G(U),\;\forall\;U\leq_oG\qbox{and}G^c\cong\bigsqcup_{g\in\cl(G)}\ls Gg$$
as $G$-space, $G^c$ is almost finite if and only if $G$ is FC and $Z(G)$ is finite, that is, if and only if $G$ is finite. 
Let us relax the requirement for $G^c$ to be almost finite, and only ask for $G$ to be a discrete $G$-space. By the above, this hold if and only if $G$ is profinite FC, and we then have $\ls Gg\cong G/C_G(g)$ for all $g\in G$.

\begin{defn}\label{def:xafg} 
Let $G$ be a profinite FC-group.
Define the category $\xAF_G$ of almost finite {\em crossed} $G$-spaces to be the category whose objects are the morphisms
$f:X\to G^c$, where $X$ is almost finite and $f$ is $G$-equivariant.
The morphisms $\phi:(f_1:X_1\to G^c)\longrightarrow(f_2:X_2\to G^c)$ between two objects in $\xAF_G$ are the morphisms $\phi\in\Hom_{\AF_G}(X_1,X_2)$ such that $f_1=f_2\phi$.
\end{defn}

\vspace{.3cm}
Let $X$ be an almost finite $G$-space.
Then a map $w:X\to G^c$ decomposes as a sum $\sqcup w_x$, where $X=\dst\bigsqcup_{x\in[G\backslash X]}G/G_x$ and $w_x:G/G_x\to G^c$ is $G$-equivariant.
Explicitly, $w_x(gG_x)=\ls g{w_x(G_x)}$, for some element $w_x(G_x)\in C_G(G_x)$. 

We define the sum and product of almost finite crossed $G$-spaces using disjoint unions and cartesian products, similarly to the case of finite groups. 
In particular, if $w_a:G/H\to G^c$ and $w_b:G/K\to G^c$ are transitive almost finite crossed $G$-spaces, their product is the almost finite crossed $G$-space
$$\bigsqcup_{g\in[H\backslash G/K]}\big(w_{a\cdot\ls gb}~:~(G/H\cap\ls gK)\longrightarrow G^c\big),$$
see \cite[Lemma 2.13(7)]{OY1}.
Two morphisms $w_a:G/H\to G^c$ and $w_b:G/K\to G^c$ are isomorphic if and only if there exists $g\in G$ such that $K=\ls gH$ and $b=\ls ga$.
With these operations, the isomorphism classes of almost finite crossed $G$-spaces form an abelian monoid. 

\begin{defn}\label{def:xb}
The {\em crossed Burnside ring of $G$} is the Grothendieck ring of the category $\xAF_G$. The elements are the isomorphism classes of virtual almost finite crossed $G$-spaces.
In particular, $1_{\wh{B^c}(G)}=[w_1:G/G\to G^c]$ and $0_{\wh{B^c}(G)}=[\emptyset\to G^c]$, where the square brackets denote isomorphism classes (which we will omit if there is no confusion), where $w_1(G)=1$ and $\emptyset$ is the initial object of the category $\AF_G$. 
\end{defn}

The following observation is immediate (cf. \cite[Section IV.8]{td}). 

\begin{lemma}
If $G$ is finite, then $\wh{B^c}(G)=B^c(G)=B(G^c)$, where $B(G^c)$ is the evaluation of the Burnside Green functor for $G$ at the $G$-set $G^c$.
\end{lemma}
  
As for finite groups, there is an injective ring homomorphism $\wh B(G)\to\wh{B^c}(G)$, defined by mapping a virtual almost finite $G$-space $X$ to $w_1:X\to G^c$, where $w_1(x)=1$ for all $x\in X$.

\cite[Lemma 2.2.2]{b-mackey} extends to our context.
\begin{lemma}\label{lem:iso-xaf}
$(v:X\to G^c)\cong(w:Y\to G^c)$ in $\xAF_G$ if and only if 
$$|\Hom_{\xAF_G}\big((w_g:G/H\to G^c),(v:X\to G^c)\big)|=|\Hom_{\xAF_G}\big((w_g:G/H\to G^c),(w:Y\to G^c)\big)|,$$
for all $(w_g:G/H\to G^c)\in\xAF_G$.
\end{lemma} 

\begin{proof}
Given almost finite $G$-spaces $X$ and $Y$, then $X\cong Y$ in $\AF_G$ if and only if $|X^U|=|Y^U|$ for all $U\leq_oG$.
Write $X=\sqcup G/G_x$ and $v=\sqcup v_{a_x}$, where $a_x\in C_G(G_x)$, and $x$ runs through a set of representatives of the $G$-orbits of $X$. 
Similarly, write $Y=\sqcup G/G_y$ and $w=\sqcup w_{b_y}$. If $[w_g:G/H\to G^c]\in\wh{B^c}(G)$, then
\begin{multline*}\Hom_{\xAF_G}\big((w_g:G/H\to G^c),(v:X\to G^c)\big)=\\
=\bigsqcup_{x\in[G\backslash X]}\Hom_{\xAF_G}\big((w_g:G/H\to G^c),(v_{a_x}:G/G_x\to G^c)\big)\end{multline*}
and similarly for $Y$. Note that these are finite sets because $\Hom_{\AF_G}(G/H,X)\cong X^H$, via the correspondence $(\varphi:G/H\to X)\mapsto\varphi(H)$, is a finite set, for all $H\leq_oG$ and for all $X\in\AF_G$.
Now,
$\Hom_{\xAF_G}\big((w_g:G/H\to G^c),(v_{a_x}:G/G_x\to G^c)\big)$ is the subset of $\Hom_{\AF_G}(G/H,G/G_x)=\big\{m_s:H\mapsto sG_x\mid H\leq\ls sG_x\big\}$ 
formed by the  almost finite crossed $G$-spaces such that, if $H\leq\ls sG_x$, then $w_g(H)=g=\ls sa_x=w_{a_x}(sG_x)$. 
The cardinality of these two sets of homomorphisms coincide if and only if the almost finite crossed $G$-spaces $(v:X\to G^c)$ and $(w:Y\to G^c)$ have the same number of $G$-orbits of the same type.
\end{proof}

In \cite[Section 2]{ds}, the authors prove that $\wh B(G)$ is a complete topological ring isomorphic to $\dst\prlim{N\noo G}B(G/N)$ via the ring homomorphism induced by the product of the fixed point maps $\fix_N:\wh B(G)\to B(G/N)$ defined below.
More generally, let $U\leq_oG$ and let $X\in\AF_G$. 
Then $N_G(U)$ acts on the finite set of $U$-fixed points $X^U$. Indeed, for all $x\in X^U$, all $u\in U$ and all $g\in N_G(U)$, we have $u(gx)=g((u^g)x)=gx$.  
Since $U$ acts trivially on $X^U$, we can regard $X^U$ as a finite $N_G(U)/U$-set. 
Since $(X\sqcup Y)^U=X^U\sqcup Y^U$, $(X\times Y)^U=X^U\times Y^U$ and $X\cong Y\Longrightarrow X^U\cong Y^U$, for any subgroup $U$ of $G$ and any $G$-spaces $X$ and $Y$, this function extends to a ring homomorphism 
$\wh B(G)\to B(N_G(U)/U)$, for all $U\leq_oG$.
Now, let $N\noo G$ and let $V\leq_oG$. Define
\begin{equation}\label{eq:fix}
\fix_N(G/V)=(G/V)^N=\left\{\begin{array}{ll}
G/V&\qbox{if $N\leq V$.}\\
\emptyset&\qbox{otherwise,}\end{array}\right.
\end{equation}
A routine exercise shows that the maps $\fix_N$ are surjective ring homomorphisms. 
Each such map has a section, called {\em inflation}, $\Inf_{G/N}^G:B(G/N)\to\wh B(G)$,  which sends a finite $G/N$-set to itself, regarded as an almost finite $G$-space on which $N$ acts trivially. 
Define 
$$\fix=\prod_{N\noo G}\fix_N~:~\wh B(G)\longrightarrow\prod_{N\noo G}B(G/N).$$ 
This is the injective ring homomorphism used in \cite[Section 2]{ds} to show that $\wh B(G)\cong\prlim{N\noo G}B(G/N)$ is a complete topological ring.
In this topology, a basis of open ideals is $\{\ker(\fix_N)\mid N\noo G\}$.

\vspace{.3cm}
Let $G$ be a profinite FC-group and let $(w_a:G/U\to G^c)$ be a transitive almost finite crossed $G$-space, where $U\leq_oG$ and $a\in C_G(U)$.
For $N\noo G$,  the fixed point map $\fix_N:\wh B(G)\to B(G/N)$ induces a ring homomorphism $\xfix_N:\wh{B^c(G)}\to B^c(G/N)$, where
$$\fix_N(w_a:G/U\to G^c)=\left\{\begin{array}{ll}
(w_{aN}:G/U\to(G/N)^c)&\qbox{if $N\leq U$, or}\\
0_{B^c(G/N)}&\qbox{otherwise.}\end{array}\right.$$
Note that $\xfix_N$ is neither injective nor surjective. 

If $N_2,N_1\noo G$ with $N_2\leq (U\cap N_1)$, then 
\begin{align*}
\xfix_{N_1/N_2}&:B^c(G/N_2)\to B^c(G/N_1)\\
\xfix_{N_1/N_2}(w_{aN_2}:G/U\to (G/N_2)^c)&=\left\{\begin{array}{ll}
(w_{aN_1}:G/U\to(G/N_1)^c)&\qbox{if $N_1\leq U$, or}\\
0_{B^c(G/N_1)}&\qbox{otherwise.}\end{array}\right.
\end{align*}
Hence 
$$\xfix=(\xfix_N)_{N\noo G}: \wh{B^c}(G)\longrightarrow\prod_{N\noo G}\wh{B^c}(G/N)\qbox{is a ring homomorphism,}$$
and, given $N_1,N_2\noo G$ with $N_2\leq N_1$, we have 
$$\xfix_{N_1}=\xfix_{N_1/N_2}\xfix_{N_2}.$$

We aim to show that $\wh{B^c}(G)\cong\prlim{N\noo G}B^c(G/N)$ (cf. \cite[Definition 1.1.3 and Proposition 1.1.4]{wilson}). 
That is, we want to show that $\wh{B^c}(G)$ is isomorphic to the subring of $\dst\prod_{N\noo G}B^c(G/N)$ formed by all the elements of the form
$\dst\big(w_N:x_N\to (G/N)^c\big)_{\!N\noo G}\in\prod_{N\noo G}B^c(G/N)$ with
$$\xfix_{N_1/N_2}(w_{N_2}:X_{N_2}\to G^c)=(w_{N_1}:X_{N_1}\to(G/N_1)^c),$$
for all $N_1,N_2\noo G$ with $N_2\leq N_1$.

By \cite[Section 2]{ds}, we know that $\fix:\wh B(G)\to\prod_{N\noo G}B(G/N)$ is an injective ring homomorphism, and that $\fix(\wh B(G))\cong\prlim{N\noo G}B(G/N)$.

For the injectivity of $\xfix$, suppose that $\xfix(w:X\to G^c)=(0_{B^c(G/N)})_{N\noo G}$. 
Then $X^N=0_{B(G/N)}$ for all $N\noo G$, which forces $X=0_{\wh B(G)}$ too, by injectivity of $\fix$. 
Since $0_{\wh B(G)}=[\emptyset]$ is the initial object in the category $\AF_G$, there is a unique almost finite crossed $G$-space with domain $0_{\wh B(G)}$, it follows that $[w:X\to G^c]=0_{\wh{B^c}(G)}$.

Let now $\dst\big(w_N:X_N\to(G/N)^c\big)_{\!N\noo G}\in\prlim{N\noo G}B^c(G/N)$ be a nonzero element. 
The sequence of the domains produces a unique element $X\in\wh B(G)$. Suppose that $X$ is the isomorphism class of 
$\sum_{U\in\CO_G}\lambda_UG/U$, where $\CO_G$ denotes a set of representatives of the conjugacy classes of open subgroups of $G$, and the $\lambda_U$ are integers. We can then write 
$$w_N=\sum_{\overset{U\in\CO_G}{N\leq U}}\sum_{1\leq i\leq|\lambda_U|}w_{a_{U/N,i}},\qbox{with}G/U=(G/N)\big/(U/N),$$
and where $a_{U/N,i}\in C_{G/N}(U/N)$, for all $1\leq i\leq|\lambda_U|$, and all $U\in\CO_G$ with $N\leq U$, $N\noo G$. 
By convention, if $\lambda_U=0$, then $\sum_{1\leq i\leq|\lambda_U|}w_{a_{U/N,i}}=0$.

Note that if $N_1,N_2\noo G$ with $N_2\leq N_1$, then $C_{G/N_2}(U/N_2)N_1\big/N_1\leq C_{G/N_1}(U/N_1)$, via the quotient map $G/N_2\to G/N_1$.
We can pick $a_{U,N,i}\in G$ such that $a_{U,N,i}N/N=a_{U/N,i}$ for $N\noo G$, and our definition of $\xfix$ implies that $a_{U,N_2,i}N_1=a_{U,N_1,i}N_1$.
The elements $a_{U,N,i}$ satisfy $[a_{U,N,i},U]\subseteq N$.
Hence, for $U\in\CO_G$ with $\lambda_U\neq0$, and for $1\leq i\leq|\lambda_U|$, let
$$\mathbf a_{U,i}=\bigcap_{\overset{N\noo G}{N\leq U}}a_{U,N,i}N.$$
Then $\mathbf a_{U,i}\neq\emptyset$ is a closed subset of $G$ (cf. \cite[Proposition 1.1.4]{RZ}), which consists of a single element $a_{U,i}$.
Indeed, suppose that $a,b\in\mathbf a_{U,i}$. That is, $a,b\in a_{U,N,i}N$, or equivalently, $b^{-1}a\in N$ for all $N\noo G$, which forces $a=b$ because $\cap_{N\noo G}N=1$.
Therefore $\mathbf a_{U,i}=\{a_{U,i}\}$, where
$$a_{U,i}\in \bigcap_{\overset{N\noo G}{N\leq U}}\{g\in G\mid [g,U]\subseteq N\}.$$
Putting $a_{U,N,i}=1$ if $N\not\leq U$, we have
$(a_{U,N,i}N)_{N\noo G}\in\prlim{N\noo G}C_{G/N}(UN/N)=C_G(U)$, we conclude that $a_{U,i}\in C_G(U)$
(cf. \cite[Exercise 0.4(2)]{wilson}).

Consequently, $(w_{a_{U,i}}:G/U\to G^c)\in\xAF_G$, and
$$\xfix(w_{a_{U,i}}:G/U\to G^c)_N=\left\{\begin{array}{ll}
[w_{a_{U/N,i}}:G/U\to(G/N)^c]\in B^c(G/N)&\:\hbox{if $N\leq U$}\\
0_{B^c(G/N)}&\;\hbox{otherwise,}\end{array}\right.$$
saying that 
$\xfix(w_{a_{U,i}}:G/U\to G^c)\in\prlim{N\noo G}B^c(G/N)$. We have thus proved the following.

\begin{prop}\label{prop:xb-top-ring}
Let $G$ be a profinite FC-group. Then $\xfix$ induces a ring isomorphism 
$$\wh{B^c}(G)\stackrel{\cong}{\longrightarrow}\prlim{N\noo G}B^c(G/N).$$
\end{prop}

\vspace{.3cm}
The crossed Burnside ring of a profinite FC-group has some of the properties similar to those of the crossed Burnside ring of a finite group. 
Let $R$ be a commutative ring, and write $\wh{B^c_R}(G)=R\otimes_\Z\wh{B^c}(G)$.
Given $U\leq_oG$ and $a\in C_G(U)$, we have $\dst\sum_{g\in [N_G(U)/U]}\ls ga\in Z(RC_G(U))$, since $\ls ga\in C_G(\ls gU)=C_G(U)$ for all $g\in N_G(U)$. As in \cite[Section 2.3]{b-mackey}, we obtain a ring homomorphism: 
$$z_U:\wh{B^c_R}(G)\longrightarrow Z(RC_G(U)),\quad
z_U(w:X\to G^c)=\sum_{\overset{x\in[G\backslash X]}{U\leq_GG_x}}\sum_{g\in[N_G(G_x)/G_x]}\ls ga_x,$$
where 
$$(w:X\to G^c)=\dst\bigsqcup_{x\in[G\backslash X]}(w_{a_x}:G/G_x\to G^c)$$ 
is an almost finite crossed $G$-space. Here, $a_x\in C_G(G_x)$ for all $x$, and the notation $U\leq_GG_x$ means that there exists $h\in G$ such that $U\leq\ls hG_x$.
The map $z_U$ extends to virtual almost finite crossed $G$-spaces, and since $|G:U|<\infty$, the above sums are finite.
Therefore $z_U$ is well defined, and we obtain a ring homomorphism
$$\zeta:\wh{B^c_R}(G)\longrightarrow\prod_{U\in\CO_G} Z(RC_G(U)),\quad\zeta(\hat w)=\big(z_U(\hat w)\big)_{U\in\CO_G},\;\forall\;\hat w\in\wh{B^c_R}(G),$$
where $\CO_G$ denotes a set of representatives of the conjugacy classes of open subgroups of $G$.
The same argument as in \cite[Lemma 2.3.2]{b-mackey} shows the following (for the proof, we now use $K\leq_oG$ with $|G:K|$ minimal such that $\hat w=\sum_i\lambda_U(w_{a_U}:G/U\to G^c)$ has a nonzero $\lambda_K$).

\begin{lemma}\label{lem:brauer-inj}
If $R$ is torsionfree, then $\zeta$ is injective. Consequently, we obtain a mapping $\dst\spec(\prod_{U\in\CO_G} Z(RC_G(U)))\to\spec(\wh{B^c_R}(G))$.
\end{lemma}
Note that the ring extension $\wh B(G)\subset\wh{B^c}(G)$ is not algebraic, and therefore the mapping in Lemma~\ref{lem:brauer-inj} need not be surjective.


\section{Idempotents of $\wh B(G)$}\label{sec:spec}

Let $G$ be a profinite group. 
We draw on the properties of the ring homomorphisms $\fix_N$ and $\Inf_{G/N}^G$ defined in Section~\ref{sec:xb} in order to investigate the relationships between the idempotents of $\wh B(G)$ with those of the Burnside rings of the finite quotient groups of $G$.

If $G$ is finite, Dress proved that $G$ is soluble if and only if the prime ideal spectrum of $B(G)$ is connected, i.e. the only idempotents of $B(G)$ are $0$ and $1$ (cf. \cite[Section 7.5, Corollary]{jacobson}). 
(By {\em ideal}, we mean an ideal that is closed in the topology of $\wh B(G)$ defined by taking $\{\ker(\fix_N)\mid N\noo G\}$ as open neighbourhood basis of $0\in\wh B(G)$.)
This result extends to profinite groups and $\wh B(G)$ in the following way.

\begin{prop}\label{prop:prosoluble}
Let $G$ be a profinite group. Then $G$ is prosoluble if and only if the prime ideal spectrum of $\wh B(G)$ is connected, i.e. the only idempotents of $\wh B(G)$ are $0$ and~$1$.
\end{prop}

\begin{proof}
We know that the result holds for finite soluble groups. Let $G$ be a prosoluble profinite group, i.e. $G/N$ is soluble for all $N\noo G$. 
Suppose that $e=e^2\in\wh B(G)$.
Since $\fix$ is a ring homomorphism, $\fix_N(e)$ is an idempotent in $B(G/N)$, and therefore $\fix_N(e)\in\{0,1\}$, for all $N\trianglelefteq_oG$.

Since $\fix$ is injective $\fix_N(e)=0$ for all $N\trianglelefteq_oG$ if and only if $e=0$ in $\wh B(G)$, i.e. $e$ is the isomorphism class of the empty set.
So, suppose that $e\neq0$. Then there must be some open normal subgroup $N$ of $G$ such that $\fix_N(e)\neq0\in B(G/N)$. 
Since $G/N$ is a finite soluble group, we must have $\fix_N(e)=1\in B(G/N)$.
That is, $\fix_N(e)=[(G/N)\big/(G/N)]\cong[G/G]\cong[(G/M)/(G/M)]$, and it follows that $\fix_M(e)=1\in B(G/M)$ for every open normal subgroup $M$ of $G$. 
We conclude that $e=1$ in $\wh B(G)$.

Conversely, suppose that $0$ and $1$ are the only idempotents of $\wh B(G)$.
Let $N\trianglelefteq_oG$. Suppose that $e_N^2=e_N\in B(G/N)$. Then $\Inf_{G/N}^G(e_N)$ is an idempotent of $\wh B(G)$. 
By assumption, this idempotent must be either $0$ or $1$. It follows that either $e_N=0$ of $e_N=1$ in $B(G/N)$ for all $N\noo G$, and so $G/N$ is a finite soluble group.
\end{proof}

The above result leads us to investigate a possible correspondence between the (primitive) idempotents of $\wh B(G)$ and those of the Burnside rings $B(G/N)$, for $N\trianglelefteq_oG$ of the finite quotients of $G$.

First, let us recall some elementary facts in group theory. By convention, a perfect group is a nonabelian (hence nontrivial) group.

\begin{remark}\label{rem:gp-facts}
\begin{enumerate}\

\item $G$ is a perfect group if and only if $G/H$ is perfect for all $H\trianglelefteq G$.
Indeed, $G$ is perfect if and only if $G$ has no nontrivial abelian quotient, if and only if no nontrivial quotient $G/H$ of $G$ has a nontrivial abelian quotient. 

In particular, if $G$ is profinite FC, \cite[Lemma 2.6]{shalev} shows that $G'$ is finite. Since the property FC is inherited by subgroups, and since $H'\leq G'$ for all $H\leq G$, any perfect subgroup of $G$ is finite. 
More generally, for an arbitrary FC-group $G$, any perfect subgroup is torsion (since $G'$ is torsion).

\item\label{it:OpG} Let $p$ be a prime, and let $G$ be a profinite group. 
Recall that for a finite group $H$, there is a unique well-defined characteristic subgroup $O^p(H)$ which is the minimal normal subgroup of $H$ with quotient a $p$-group. 
For all $N\noo G$, let $U_N$ the characteristic open subgroup of $G$ such that $N\leq U_N\leq G$ and $G/U_N\cong (G/N)\big/O^p(G/N)$.
If $N_2\leq N_1$ are open normal subgroups of $G$, then $G/U_1\cong(G/N_2)\big/(U_1/N_2)$ is a (finite) $p$-group, quotient of $G/N_2$, where $U_i=U_{N_i}$. 
Therefore, $G/U_2\onto G/U_1$, and we obtain an inverse system of finite $p$-groups $\{G/U_i,\;G/U_j\onto G/U_i\;(\forall\;N_j\leq N_i), \;N_i,N_j\noo G\}$.
Let $\ovl G=\prlim{N\noo G}G/U_N$ be the inverse limit, and $\theta_p:G\to\ovl G$ the quotient map induced by the projections $G/N\to G/U_N$. 
Note that $\theta_p$ is well defined since the squares $\xymatrix{G/N_2\ar[r]\ar[d]&G/N_1\ar[d]\\G/U_2\ar[r]&G/U_1}$, where the maps are the quotient maps, commute for all $N_1,N_2\noo G$ with $N_2\leq N_1$.
We define
$$O^p(G)=\ker(\theta_p)=\bigcap_{N\noo G}U_N.$$
Then, $O^p(G)$ is a closed characteristic subgroup of $G$ with the property that any pro-$p$ quotient group of $G$ is a quotient of $G/O^p(G)$.  

\item Let $H\leq G$ be a finite subgroup of a residually finite group $G$.
For each $x\in H$, there exists $N_x\trianglelefteq_oG$ such that $x\notin N_x$. 
Let $N_H=\dst\bigcap_{x\in H}N_x$. Then $N_H\noo G$ and $N_H\cap H=1$. Hence, there are infinitely many open normal subgroups of $G$ which do not meet $H$. 
In particular, if $G$ is profinite FC and $H$ is a perfect subgroup of $G$, then the set 
$$\mathcal N_H=\{N\trianglelefteq_oG\mid |H\cap N|=1\}$$
is a filter base for $G$, that is, $\mathcal N_H$ is a family of open normal subgroups of $G$ such that:
\begin{itemize}
\item[(i)] for all $N_1,N_2\in\mathcal N_H$, there exists $N_3\in\mathcal N_H$ with $N_3\leq N_1\cap N_2$, and
\item[(ii)] $\dst\bigcap_{N\in\mathcal N_H}N=\{1\}$.
\end{itemize}
Indeed we note that we have the stronger condition that for any $N_1\trianglelefteq_oG$ and for any $N_2\in\mathcal N_H$, then $N_1\cap N_2\in\mathcal N_H$, from which follows that $\dst\bigcap_{N\in\mathcal N_H}N=\bigcap_{N\trianglelefteq_oG}N=\{1\}$.
\end{enumerate}
\end{remark}

Remark~\ref{rem:gp-facts}~\eqref{it:OpG} would lean towards a definition of idempotents in $\wh B_{\Z_p}(G)$, if we tried to extend the result from finite groups.
However, as we shall shortly see, our methods only allow us to define idempotents indexed by open subgroups of $G$, but we cannot expect the $p$-perfect subgroups of a profinite group to be all open.
We thus leave this question aside, referring the reader to Proposition~\ref{prop:prim-idemp} as a starter towards generalising further our results.
We close this parenthesis on some remarks about groups with the following observation.

Let $G$ be a profinite group, and define the set of (closed) subgroups of $G$
$$\mathcal P=\{1<H\leq G\mid H'=H\}.$$
Write $[\mathcal P]$ for a set of representatives of the $G$-conjugacy classes of perfect subgroups of $G$.
For $n\in\N$, define inductively the (closed) derived series $G=G^{(1)}\geq G^{(2)}\geq G^{(3)}\geq\dots$ for $G$, where we define $G^{(1)}=G$ and $G^{(i)}=\ovl{[G^{(i-1)},G^{(i-1)}]}$ for all $i\geq2$. 
Since the series is monotone decreasing, if some $G^{(n)}$ is finite, then the series converges and we have a well defined subgroup $G^{(\infty)}=\dst\bigcap_{n\in\N}G^{(n)}$.

\begin{lemma}\label{lem:prosol}
Let $G$ be a profinite group.
Then $\mathcal P\neq\emptyset$ if and only if $G$ is not prosoluble, if and only if $G^{(\infty)}$ exists and is nontrivial, that is, the derived series converges to a nontrivial subgroup of $G$. 
\end{lemma}

\begin{proof}
First, note that $\mathcal P\neq\emptyset$ if and only if $G$ is not prosoluble, since $H\in\mathcal P$ if and only if for all $N\noo G$ such that $H\not\leq N$, then $G/N$ is a finite group with a perfect nontrivial subgroup $HN/N$. Hence, we need to show that $\mathcal P\neq\emptyset$ if and only if $G^{(\infty)}$ exists and is nontrivial.

If $G^{(\infty)}$ exists and is nontrivial, then $G^{(\infty)}\in\mathcal P$. 
Conversely, suppose that $\mathcal P\neq\emptyset$.
Let $U=\ovl{\langle H\mid H\in\mathcal P\rangle}$. 
Note that $1\neq U$ is characteristic in $G$ since $\mathcal P$ is closed under $G$-conjugation and since the image of any perfect subgroup of $G$ by an automorphism of $G$ is again a perfect subgroup of $G$. 
Moreover, $U'\geq\ovl{\langle H'\mid H\in\mathcal P\rangle}=U$ shows that $1\neq U\in\mathcal P$ is perfect.
The assertion follows from the observation that $G/U$ is soluble. 
Indeed, any perfect subgroup $H/U$ of $G/U$, with $U\leq H\leq G$, satisfies $H=H'U=H'U'=H'$, where the first equality holds because $H/U=(H/U)'=H'U/U$. 
Hence, $H\in\mathcal P$, which implies that $H=U$.
Therefore, $G/U$ is profinite and soluble, and we must have $1\neq U=G^{(\infty)}$ as required.
\end{proof}

The set $[\mathcal P]$ is useful in the description of the integral idempotents of the Burnside ring of a finite group $G$. Indeed, if $G$ is a finite group,
the primitive idempotents of the Burnside $\Q$-algebra $B_\Q(G)$ of $G$ are indexed by the conjugacy classes of subgroups $H$ of $G$, and have the form \cite{gluck}:
$$e_H=\sum_{1\leq K\leq H}\frac{\mu(K,H)}{|N_G(H):K|}G/K$$
where $\mu(-,-)$ denotes the M\"obius inversion formula, and the sum is over all the subgroups of $H$. 
In $B_\Q(G)$, we have $e_H^2=e_H$ and $e_H$ is characterised by $|(e_H)^K|=1$ if and only if $K$ is $G$-conjugate to $H$, and $|(e_H)^K|=0$ otherwise.

In the (integral) Burnside ring $B(G)$, the set
$$\{f_H=\sum_Ke_K\mid H\in[\mathcal P]\cup\{1\}\}$$
is a complete set of primitive pairwise orthogonal idempotents,
where $K$ runs through a set of representatives of the conjugacy classes of subgroups of $G$ such that $K^{(\infty)}$ is $G$-conjugate to $H$, for all $H\in[\mathcal P]\cup\{1\}$.

\vspace{.3cm}

Suppose now that $G$ is profinite.
We generalise Gluck's idempotent formula to $\wh B_{\Q}(G)$ as follows: Let $\CO_G$ denote a set of representatives of the conjugacy classes of open subgroups of $G$. For $H\in\CO_G$, put 
$$e_H=\sum_{K\leq_oH}\frac{\mu(K,H)}{|N_G(H):K|}G/K.$$
For all $N\noo G$, we have $\fix_N(e_H)=e_{H/N}$ if $N\leq H$ and $\fix_N(e_H)=0$ if $N\not\leq H$, as element in $B_\Q(G/N)$.
Indeed, 
$$\fix_N(e_N)=\sum_{N\leq K\leq H}\frac{\mu(K,H)}{|N_G(H):K|}G/K=\sum_{N\leq K\leq H}\frac{\mu(K/N,H/N)}{|N_{G/N}(H/N):K/N|}(G/N)\big/(K/N)$$
in $B_\Q(G/N)$, since, if $N\leq H$, then $N_{G/N}(H/N)=\{gN\in G/N\mid\ls{gN}H\leq HN=H\}=N_G(H)/N$.
By definition, $(e_H)^2=e_H$, since $(\fix_N(e_H))^2=\fix_N(e_H)$ for all $N\noo G$. 
Now, if $U\leq_oG$, then $|(e_H)^U|=|(e_{H/N})^{U/N}|$ for any $N\noo G$ with $N\leq H\cap U$. 
By the case of finite groups, this number is $1$ if $U/N$ is $G/N$-conjugate to $H/N$, and $0$ otherwise. It then suffices to observe that for such $N$, $U/N$ is $G/N$-conjugate to $H/N$ if and only if $U$ is $G$-conjugate to $H$. 
It follows that $|(e_H)^U|=1$ if $U$ is $G$-conjugate to $H$ and $0$ otherwise, for all $H,U\leq_oG$. 

We have shown the following.
\begin{prop}\label{prop:prim-idemp}
Assume the above notation. 
\begin{enumerate}
\item For every $H\leq_oG$, the element $e_H=\sum_{K\leq_oH}\frac{\mu(K,H)}{|N_G(H):K|}G/K\in\wh B_\Q(G)$ is an idempotent. In particular, it need not be a finite $\Q$-linear combination of transitive finite $G$-sets.
\item The ghost map 
$$\wh B_\Q(G)\longrightarrow\Q^{\CO_G},\quad x\mapsto(|x^U|)_{U\in\CO_G}$$
maps the set $\{e_H\mid H\in\CO_G\}$ to a canonical basis of the ghost $\Q$-algebra. That is, $e_H\longmapsto(\delta_{U,H})_{U\in\CO_G}$, where $\delta_{U,H}=1$ if $U$ is conjugate to $H$ and is $0$ otherwise.
\end{enumerate}
\end{prop}

\begin{example}
Let $G=\Z_p$ for a prime $p$. Then, for all $n\geq0$, 
$$e_{p^nG}=\frac1{p^n}G/p^nG-\frac1{p^{n+1}}G/p^{n+1}G~,$$
since for nonnegative integers $m\leq n$, we have $\mu(p^mG,p^nG)=1$ if $n=m$, $\mu(p^mG,p^nG)=-1$ if $m+1=n$, and $\mu(p^mG,p^nG)=0$ otherwise.

If $G=\wh{\Z}$, then $e_G=\dst\sum_n\frac{\mu(n\wh{\Z},\wh{\Z})}n\wh{\Z}/n\wh{\Z}$, where $n$ runs through all the integers which factorise into a product of distinct primes.
\end{example}

Dress's result \cite[Proposition 2]{dress-sol} does not extend as such to obtain a complete list of the integral primitive idempotents in $\wh B(G)$.
If $H$ is a perfect subgroup of $G$, it need not contain any open subgroup, and there may be infinitely many subgroups $K$ of $G$ such that $K^{(\infty)}$ is $G$-conjugate to $H$.

In particular, if $G$ is an infinite profinite FC-group, then any perfect subgroup $H$ of $G$ is finite, and therefore the set $\mathcal P$ is a subset of the finite subgroups of $G$, each of which possessing finitely many conjugates. 
Since $Z(G)$ has finite index in $G$, there are infinitely many subgroups $K\leq G$ such that $K^{(\infty)}$ is $G$-conjugate to $H$. 
But the elements $e_H$ introduced above are only defined for subgroups of finite index. Our methods lead, for instance to idempotents inflated from $B(G/Z(G))$.
We conclude this section with an example.

\begin{example}
Let $G=A_5\times\wh{\Z}$. Then $G$ is profinite FC and $H=A_5\times\{1\}$ is the unique nontrivial perfect subgroup of $G$. 
Note that for any $H\leq U\leq G$ we have $U'=U^{(\infty)}=H$.
Let
$$e_H=\Inf_{G/Z(G)}^G\big(A_5/A_5-A_5/A_4-A_5/D_{10}-A_5/D_6+A_5/C_3+2A_5/C_2-A_5/1\big),$$
where we write $A_5=G/Z(G)\cong H$, and we use the obvious identifications of the subgroups of $H$.
We might expect $(e_H)^2=e_H\in\wh B(G)$ to be a summand of $e_G$.
\end{example}


\section{Action of almost finite crossed $G$-spaces on Mackey functors for profinite groups}\label{sec:xb-mackey}

Let $G$ be a finite group and let $R$ be a commutative ring. Each crossed $G$-set acts on the category $\Mack_R(G)$ of Mackey functors for $G$ over $R$, producing a natural transformation of the identity morphism in $\Mack_R(G)$.
This property has been used in \cite{b-mackey} to obtain a ring homomorphism from the crossed Burnside ring of $G$ to the centre of the Mackey $R$-algebra for $G$ over $R$. 
Mackey functors have been extended from finite to profinite groups, taking some different perspectives depending on the objective(s) of the authors (\cite{bb,nakaoka} and \cite[Section 5]{ds}). 
In the present section, we show that the almost finite crossed $G$-spaces act on a category of Mackey functors. We follow in parallel \cite{bb} and \cite{ds}, specialising their perspective to our context.

Throughout, let $G$ be a profinite group and let $R$ be a commutative ring. 
We build on Section~\ref{sec:profinite}.

\begin{defn}\label{def:xafgf}
Let $\AF_G^r$ be the subcategory of $\AF_G$ with the same objects as $\AF_G$, and morphisms $f:X\to Y$ are the almost finite morphisms such that the fibres $f^{-1}(y)$ are finite, $\forall\;y\in Y$.
\end{defn}
The categories $\AF_G^r$ and $\AF_G$ of discrete $G$-spaces are introduced and used in \cite{bb,ds,nakaoka}. 
By contrast, in \cite{mpn}, the authors consider $G$-spaces $X$, for a discrete group $G$, such that each point stabiliser is a finite subgroup of $G$, and such that $X$ has finitely many $G$-orbits. 

Let us record some useful observations.
\begin{remark}\label{rem:g-spaces}\

\begin{enumerate}
\item If $f\in\Hom_{\AF_G}(X,Y)$, then $f^{-1}(y)$ is an almost finite $G_y$-space for all $y\in Y$, and $G_x\leq G_y$ for all $x\in f^{-1}(y)$.
\item If $G$ is finite, then $\AF_G=\AF_G^r$.
\item If $G$ is infinite, then $\AF_G^r$ has no terminal object, since $\Hom_{\AF_G^r}(X,G/G)\neq\emptyset$ if and only if $X$ is finite. 
\end{enumerate}
\end{remark}

We introduce two kinds of Mackey functors for a given profinite group $G$ (compare with \cite[Definition 2.6]{bb} and \cite[Section 5]{ds}).
\begin{defn}\label{def:mackey-cat}
A {\em Mackey functor for $G$} is an additive functor $M=(M_*,M^*):\AF_G\times\AF_G\to\Ab$
with $M_*$ covariant and $M^*$ contravariant, subject to the following axioms.
\begin{itemize}
\item[{\bf(MF1)}] $M_*(X)=M^*(X)$ for every almost finite $G$-space $X$. Thus we write simply $M(X)$. 
\item[{\bf(MF2)}] If $\dst\bigsqcup_iX_i$ is almost finite, then the natural inclusions $X_i\to\dst\bigsqcup_iX_i$ induce an isomorphism $\dst M(\bigsqcup_iX_i)\cong\prod_iM(X_i)$ of abelian groups.
\item[{\bf(MF3)}]
For any pull back diagram of almost finite $G$-spaces
$$\xymatrix{X\ar[r]^\alpha\ar[d]_\beta&Y\ar[d]^\gamma\\  
Z\ar[r]_\delta&W}\;\hbox{in $\AF_G$, the diagram}\;
\xymatrix{M(X)\ar[d]_{M_*(\beta)}&M(Y)\ar[d]^{M_*(\gamma)}\ar[l]_{M^*(\alpha)}\\  
M(Z)&M(W)\ar[l]^{M^*(\delta)}}\;\hbox{commutes in $\Ab$.}$$
\end{itemize}

A {\em restricted Mackey functor for $G$} is an additive functor $M=(M_*,M^*):\AF_G^r\times\AF_G\to\Ab$
with $M_*$ covariant and $M^*$ contravariant, subject to the same axioms {\bf(MF1, MF2)}, and with a {\em restricted} variant of {\bf(MF3)}, where the vertical maps in the left hand side pull back diagram are morphisms in $\AF_G^r$.

We let $\Mack(G)$ (resp. $\Mack^r(G)$) denote the category of (restricted) Mackey functors for $G$, whose objects are the (restricted) Mackey functors for $G$, and the morphisms are the natural transformations of functors.
Given a commutative ring $R$, we set $\Mack_R(G)=R\otimes_\Z\Mack(G)$ and call it the category of Mackey functors for $G$ over $R$.
\end{defn}

\begin{remark}\label{rem:pullback}
Recall that a pull back of $G$-spaces encodes the Mackey formula: If
$$\xymatrix{X\ar[r]^\alpha\ar[d]_\beta&G/H\ar[d]^\gamma\\ G/K\ar[r]_\delta&G/L},$$
is a pull back with $H,K,L$ closed subgroups of $G$, and, to simplify, assume that the maps $\gamma$ and $\delta$ are induced by inclusions of subgroups $H,K\hookrightarrow L$, then $X$ is of the form
$$X\cong\bigsqcup_{u\in[H\backslash L/K]}G/(H\cap\ls uK).$$
In particular, the double coset space $[H\backslash L/K]$ is a discrete space if and only if at least one of $H$ or $K$ is an open subgroup of $L$. 
\end{remark}

In \cite[Section 5]{ds}, the authors show that the {\em Burnside functor} $\wh B:=\wh{B^G}~:~\AF_G\to\Ab$, where 
$$\wh B(X)=\Hom_{\AF_G}(-,X)=:\{f:Y\to X\;\hbox{in}\;\AF_G\}$$ 
satisfies the axioms of Mackey functors, without any finiteness assumption on the fibres of the maps. In their setting, the covariant part $\wh B_*$ is given by the composition of maps:
If $\phi\in\Hom_{\AF_G}(X,Z)$, then $\wh B_*(\phi):\wh B(X)\to\wh B(Z)$ is given by $\wh B_*(\phi)\big(f:Y\to X\big)=\big(\phi f:Y\to Z)$. 
The contravariant part $\wh B^*(\phi):\wh B(Z)\to\wh B(X)$ is given by the pull back: For $f\in\Hom_{\AF_G}(Y,Z)$,
$$\xymatrix{U\ar[r]\ar[d]_{\wh B^*(\phi)(f)}&Y\ar[d]^f\\X\ar[r]^\phi&Z}$$
By contrast, if $V$ is an $RG$-module, for some commutative ring $R$, the fixed point module functor $\FP_V$ and the fixed quotient module functor $\FQ_V$ are not Mackey functors for $G$ over $R$, but restricted Mackey functors. 
Recall that they are defined on an almost finite $G$-space $X$ by
$$\FP_V(X)=\prod_{x\in[G\backslash X]}V^{G_x}\qbox{and}
\FQ_V(X)=\prod_{x\in[G\backslash X]}V_{G_x},$$
where $V_H=V/\langle hv-v\mid h\in H,\;v\in V\rangle$ denotes the $H$-coinvariants of $V$. If $f\in\Hom_{\AF_G^r}(X,Y)$, then the image of
$$(\FP_V)_*(f): \prod_{x\in[G\backslash X]}V^{G_x}\longrightarrow\prod_{y\in[G\backslash Y]}V^{G_y}$$
in the $V^{G_y}$ coordinate consists of elements of the form $\dst\sum_{x\in f^{-1}(y)}\sum_{g\in[G_y/G_x]}gv_x$ for elements $v_x\in V^{G_x}$, for all $x\in f^{-1}(y)$. 
This is well defined if and only if $f^{-1}(y)$ is a finite set. 
The contravariant Mackey functor $(\FP_V)^*$ is induced by the inclusions of fixed points $V^{G_y}\hookrightarrow V^{G_x}$ for all $x\in f^{-1}(y)$.
Similarly, the image of
$$(\FQ_V)_*(f):\prod_{x\in[G\backslash X]}V_{G_x}\longrightarrow\prod_{y\in[G\backslash Y]}V_{G_y}$$
in the $V_{G_y}$ coordinate consists of elements of the form $\dst\sum_{x\in f^{-1}(y)}\sum_{g\in[G_y/G_x]}\ovl{v_x}$, where $\ovl{v_x}\in V_{G_y}$ is the image of $v_x\in V_{G_x}$ via the quotient map 
$\xymatrix{V_{G_x}\ar@{->>}[r]&V_{G_y}}$ for all $x\in f^{-1}(y)$. Again, this is well defined if and only if $f^{-1}(y)$ is a finite set. 
The contravariant Mackey functor $(\FQ_V)^*$ is induced by the inclusions of $R$-modules $V_{G_y}\hookrightarrow V_{G_x}$ for all $x\in f^{-1}(y)$.

\vspace{.3cm}
A key observation in \cite{tw}, extended in \cite{b-mackey} (referring to the original work of Yoshida), is that, if $G$ is a finite group, then the crossed $G$-sets act on the category of Mackey functors. 
We now generalise this action to profinite FC-groups and almost finite crossed $G$-spaces. 

Let $(f:X\to G^c)\in\xAF_G$ and let $Y\in\AF_G$. Define the mappings:
\begin{align*}
\pi_Y,\tau^f_Y:~X\times Y&\longrightarrow Y,\\
\tau^f_Y(x,y)&=f(x)y\qbox{and}\\
\pi_Y(x,y)&=y,\qbox{for all $(x,y)\in X\times Y$,}
\end{align*}
where we have abbreviated the notation for convenience ($\pi^{X\times Y}_Y$ and $\tau^{(f:X\to G^c)}_Y$ would be more precise than $\pi_Y$ and $\tau^f_Y$, respectively).
Clearly, both are continuous and $\pi_Y$ is $G$-equivariant ($G$ acts on $X\times Y$ diagonally).
The map $\tau^f_Y$ is $G$-equivariant too, since for all $g\in G$ and all $(x,y)\in X\times Y$, we have
$$\tau^f_Y\big(g\cdot(x,y)\big)=\tau^f_Y(gx,gy)=
f(gx)gy=\ls g{f}(x)gy=gf(x)y=g\tau^f_Y(x,y).$$
The fibres $\pi_Y^{-1}(y)$ and $(\tau^f_Y)^{-1}(y)$ are subsets of the almost finite $G$-space $X\times Y$, and therefore they are almost finite $G_y$-spaces for all $y\in Y$. 
Thus $\pi_Y$ and $\tau^f_Y$ are morphisms in $\AF_G$. 
Note that, if $f(x)=1$ for all $x\in X$, then $\tau^f_Y=\pi_Y$.

\vspace{.3cm}
Now, let $M$ be a Mackey functor for $G$ over $R$. Consider the composition
$$\eta^f_Y=M_*(\tau^f_Y)M^*(\pi_Y)~:~M(Y)\longrightarrow M(Y).$$ 
By definition of Mackey functors, this composition is an $R$-module homomorphism. 
Given $\alpha\in\Hom_{\AF_G}(Y,Y')$, the diagrams of $R$-modules and homomorphisms
$$\xymatrix{M(Y)\ar[r]^{\eta^f_Y}\ar[d]_{M_*(\alpha)}&M(Y)\ar[d]^{M_*(\alpha)}\\
M(Y')\ar[r]^{\eta^f_{Y'}}&M(Y')}\qbox{and}\xymatrix{M(Y)\ar[r]^{\eta^f_Y}&M(Y)\\
M(Y')\ar[r]^{\eta^f_{Y'}}\ar[u]^{M^*(\alpha)}&M(Y')\ar[u]_{M^*(\alpha)}}\qbox{commute.}$$

It follows that \cite[Proposition 4.3]{b-mackey} holds in the present context.

\begin{prop}\label{prop:xb-action}
Let $(f:X\to G^c)\in\xAF_G$.
The map $\eta^f$ is a natural transformation of the identity functor of the category $\Mack_R(G)$. Moreover, if $(f':X'\to G^c)\in\xAF_G$, then
\begin{align*}
\eta^f+\eta^{f'}&=\eta^{f\sqcup f'},\qbox{and}\\
\eta^f\cdot\eta^{f'}&=\eta^{f\times f'},
\end{align*}
where, for any almost finite $G$-space $Y$, there are $R$-module endomorphisms of $M(Y)$,
$$(\eta^f+\eta^{f'})_Y=\eta^f_Y\oplus\eta^{f'}_Y:\xymatrix{M(Y)\ar[r]&M(X)\oplus M(X')\ar[r]&M(Y)}$$
and 
$$(\eta^f\cdot\eta^{f'})_Y=\eta^{f\times f'}_Y:\xymatrix{M(Y)\ar[r]&M(X)\otimes_RM(X')\ar[r]&M(Y)}.$$
\end{prop}

Explicitly, Proposition~\ref{prop:xb-action} states that, for a profinite FC-group $G$, the abelian monoid of almost finite crossed $G$-spaces acts on the category of Mackey functors. 
If $G$ is an arbitrary profinite group, the action extended from \cite[Section 9]{tw} remains well defined too, where
$$(X\cdot M)(Y)=\big(M_*(\pi_Y)M^*(\pi_Y)\big)(M(Y)),$$
for all almost finite $G$-spaces $X$ and $Y$, and for all Mackey functors $M$ for $G$.

The proof is routine. For instance, for the equality $\eta^f\cdot\eta^{f'}=\eta^{f\times f'}$, let $f:X\to G^c$, let $f':X'\to G^c$, let $M$ be a Mackey functor for $G$ over $R$, and let $Z\in\AF_G$. Then,
$$\xymatrix{M(Z)\ar[rr]^{M^*(\pi_Z)}\ar[d]_{M^*(\pi_Z)}&&M(X\times Z)\ar[rr]^{M_*(\tau^f_Z)}\ar@{.>}[dll]_{M^*(\pi_{X\times Z})}&&M(Z)\ar[d]^{M^*(\pi_Z)}\\
M(X\times X'\times Z)\ar[drrrr]_{M_*(\tau^{f\times f'}_Z)}\ar@{.>}[rrrr]^{M_*(\tau^f_{X'\times Z})}&&&&M(X'\times Z)\ar[d]^{M_*(\tau^{f'}_Z)}\\
&&&&M(Z)}$$
the dotted maps make the diagram commute, and they are obtained applying $M$ to the pull back in $\AF_G$,
$$\xymatrix{X\times X'\times Z\ar[rr]^{\pi_{X\times Z}}\ar[d]_{\tau^f_{X'\times Z}}&&X\times Z\ar[d]^{\tau^f_Z}\\
X'\times Z\ar[rr]_{\pi_Z}&&Z}.$$

\vspace{.3cm}
If instead we consider restricted Mackey functors for a profinite FC-group $G$, as in \cite{bb}, then $\xAF_G$ does not act on $\Mack^r(G)$. 
Indeed, if $(f:X\to G^c)\in\xAF_G$ and $Y\in\AF_G$, then $\pi_Y:X\times Y\to Y$ is a morphism in $\xAF_G^f$ if and only if $X$ is finite. 
Instead, $\tau^f_Y$ is a morphism in $\xAF_G^f$ if and only if, for all $y\in Y$, the set $\{(x,z)\in X\times Y\mid y=f(x)z\}$ is finite. 
Thus, only the finite crossed $G$-sets act on $\Mack^r(G)$.
This observation does not come as a surprise to us, but it raises the question of the structure and purpose of the category (or categories) of Mackey functors for a profinite groups.

\vspace{.5cm}
\noindent{\bf Acknowledgements.}\; We are sincerely grateful to Serge Bouc, Ilaria Castellano, Brita Nucinkis and Jacques Th\'evenaz for several helpful discussions on the various aspects of this project.



\begin{thebibliography}{WWWW}
\bibitem{bb} W. Bley and R. Boltje, {\em Cohomological Mackey functors in number theory}, J. Number Theory {\bf105} (2004), 1--37.
\bibitem{b-mackey} S. Bouc, {\em The $p$-blocks of the Mackey algebra}, Algebr. Represent. Theory 6 (2003), no. 5, 515--543.
\bibitem{cr} C. Curtis and I. Reiner, {\em Methods of representation theory: with applications to finite groups and orders}, Volume 1, Wiley, 1981
\bibitem{td} T. tom Dieck, {\em Transformation groups}, de Gruyter Studies in Mathematics 8, Walter de Gruyter, 1987. 
\bibitem{dress-sol} A. Dress, {\em A characterization of solvable groups}, Math Z. {\bf 110} (1969), 213--217.
\bibitem{ds} A. Dress and C. Siebeneicher, {\em The Burnside ring of profinite groups and the Witt vector construction}, Adv. in Math. {\bf70} (1988), no. 1, 87--132.
\bibitem{gluck} D. Gluck, {\em Idempotent formula for the Burnside algebra with applications to the $p$-subgroup simplicial complex}, Ill. J. Math. 25 (1981), 63--67.
\bibitem{hallp} P. Hall, {\em Periodic FC-groups}, J. LMS {\bf34} (1959), 289--304.
\bibitem{jacobson} N. Jacobson, {\em Basic Algebra II}, Second Edition, Dover Publications, 2009.
\bibitem{mpn} C. Martinez-P\'erez, B. Nucinkis, {\em Cohomological dimension of Mackey functors for infinite groups}, J. Lond. Math. Soc., II. Ser. 74, No. 2 (2006), 379--396
\bibitem{may} J. P. May, {\em A concise course in algebraic topology}, Chicago lectures in mathematics series, the University of Chicago Press, 1999
\bibitem{nakaoka} H. Nakaoka, {\em Tambara functors on profinite groups and generalised Burnside functors}, Comm. Alg. 37 (2009), 3095 -- 3151. 
\bibitem{OY1} F. Oda and T.  Yoshida, {\em Crossed Burnside Rings I.}, J.  Algebra {\bf236} (2001), 29--79.
\bibitem{RZ} L. Ribes and P. Zalesskii, {\em Profinite groups}, Springer, 2000.
\bibitem{Rob} D. Robinson, {\em A course in the theory of groups}, Second Edition, Springer, 1996.   
\bibitem{shalev} A. Shalev, {\em Profinite groups with restricted centralizers}, Proc. AMS {\bf122} number 4 (1994), 1279--1284.
\bibitem{tw} J. Th\'evenaz and P. Webb, {\em The structure of Mackey functors}, Trans. Amer. Math. Soc. 347, number 6 (1995), 1865--1961.
\bibitem{tomk} M. J. Tomkinson, {\em FC-groups}, Research Notes in Mathematics 96, Pitman Advanced Publishing Program, 1984.
\bibitem{wilson} J. Wilson, {\em Profinite groups},  London Math. Soc. Monographs New Series {\bf19}, Oxford University Press, 1998. 
\end{thebibliography}
\end{document}